\documentclass[letterpaper,  11pt,  reqno]{amsart}

\usepackage{amsmath,amssymb,amscd,amsthm,amsxtra, esint}

\usepackage[implicit=true]{hyperref}

\allowdisplaybreaks[2]

\usepackage{color}

\definecolor{gr}{rgb}   {0.,   0.69,   0.23 }
\definecolor{bl}{rgb}   {0.,   0.5,   1. }
\definecolor{mg}{rgb}   {0.85,  0.,    0.85}
\definecolor{yl}{rgb}   {0.8,  0.7,   0.}
\definecolor{or}{rgb}  {0.7,0.2,0.2}

\newtheorem{theorem}{Theorem} [section]

\newtheorem{lemma}[theorem]{Lemma}
\newtheorem{proposition}[theorem]{Proposition}
\newtheorem{remark}[theorem]{Remark}

\newtheorem*{acknowledgment}{Acknowledgment}

\DeclareMathOperator*{\supp}{supp}

\newcommand{\noi}{\noindent}
\newcommand{\Z}{\mathbb{Z}}
\newcommand{\R}{\mathbb{R}}

\newcommand{\T}{\mathbb{T}}

\let\Re=\undefined\DeclareMathOperator*{\Re}{Re}
\let\Im=\undefined\DeclareMathOperator*{\Im}{Im}

\newcommand{\E}{\mathbb{E}}

\newcommand{\N}{\mathcal{N}}
\newcommand{\NB}{\mathbb{N}}

\newcommand{\be}{\beta}
\newcommand{\dl}{\delta}

\newcommand{\nb}{\nabla}

\newcommand{\Dl}{\Delta}
\newcommand{\eps}{\varepsilon}

\newcommand{\g}{\gamma}
\newcommand{\G}{\Gamma}
\newcommand{\ld}{\lambda}

\newcommand{\s}{\sigma}

\newcommand{\ft}{\widehat}

\newcommand{\wt}{\widetilde}
\newcommand{\cj}{\overline}

\newcommand{\dd}{\partial}

\renewcommand{\l}{\ell}
\renewcommand{\o}{\omega}
\renewcommand{\O}{\Omega}

\newcommand{\les}{\lesssim}

\newcommand{\jb}[1]
{\langle #1 \rangle}

\newcommand{\ind}{\mathbf 1}

\numberwithin{equation}{section}
\numberwithin{theorem}{section}

\newcommand{\too}{\longrightarrow}

\usepackage{tikz} 
\usepackage{marginnote}
\usepackage{scalerel} 

\usetikzlibrary{shapes.misc}
\usetikzlibrary{shapes.symbols}
\usetikzlibrary{decorations}
\usetikzlibrary{decorations.markings}

\tikzset{
	dot/.style={circle,fill=black,draw=black,inner sep=0pt,minimum size=0.5mm},
	>=stealth,
	}

\tikzset{
	dot2/.style={circle,fill=black,draw=black,inner sep=0pt,minimum size=0.2mm},
	>=stealth,
	}

\tikzset{
	ddot/.style={circle,fill=white,draw=black,inner sep=0pt,minimum size=0.8mm},
	>=stealth,
	}

\tikzset{decision/.style={ 
        draw,
        diamond,
        aspect=1.5
    }}

\tikzset{dia2/.style
={diamond,fill=white,draw=black,inner sep=0pt,minimum size=1mm},
	>=stealth,
	}

\tikzset{dia/.style
={star,fill=black,draw=black,inner sep=0pt,minimum size=1mm},
	>=stealth,
	}

\tikzset{dia/.style
={diamond,fill=black,draw=black,inner sep=0pt,minimum size=1.3mm},
	>=stealth,
	}

\makeatletter
\def\DeclareSymbol#1#2#3{\expandafter\gdef\csname MH@symb@#1\endcsname{\tikz[baseline=#2,scale=0.15]{#3}}}
\def\<#1>{\csname MH@symb@#1\endcsname}
\makeatother

\newcommand{\HS}{\textit{HS}}

\begin{document}
\baselineskip = 13pt

\title[stochastic NLS with non-smooth additive noise]
{On the  stochastic nonlinear Schr\"odinger equations with non-smooth additive noise}

\author[T.~Oh, O.~Pocovnicu, and Y.~Wang]
{Tadahiro Oh, Oana Pocovnicu, and Yuzhao Wang}

\address{
Tadahiro Oh\\
School of Mathematics\\
The University of Edinburgh\\
and The Maxwell Institute for the Mathematical Sciences\\
James Clerk Maxwell Building\\
The King's Buildings\\
 Peter Guthrie Tait Road\\
Edinburgh\\ 
EH9 3FD\\United Kingdom} 

\email{hiro.oh@ed.ac.uk}

\address{
Oana Pocovnicu\\
Department of Mathematics, Heriot-Watt University and The Maxwell Institute for the Mathematical Sciences, Edinburgh, EH14 4AS, United Kingdom}
\email{o.pocovnicu@hw.ac.uk}

\address{
Yuzhao Wang\\
School of Mathematics, 
University of Birmingham, 
Watson Building, 
Edgbaston, 
Birmingham\\
B15 2TT\\
United Kingdom}

\email{y.wang.14@bham.ac.uk}

\subjclass[2010]{35Q55, 60H30}

\keywords{stochastic nonlinear Schr\"odinger equation; well-posedness; dispersive estimate}

\begin{abstract}
We study  the stochastic nonlinear Schr\"odinger equations 
with additive stochastic forcing.
By using the dispersive estimate, 
we present a simple argument, constructing a unique local-in-time solution
 with rougher stochastic forcing than those considered in the literature.

\end{abstract}



%
\maketitle
%

\vspace*{-6mm}

\section{Introduction}
\subsection{Stochastic nonlinear Schr\"odinger equations}

We consider the Cauchy problem of the following 
stochastic nonlinear Schr\"odinger equations (SNLS) with additive  noise:\footnote{Since our interest
is local in time, the defocusing/focusing nature of the equations does not play any role in this paper.
Hence, we simply consider the defocusing equations.}
\begin{equation}
\begin{cases}\label{SNLS1}
i \partial_t u =   \Delta u  -  |u|^{p-1} u  + \phi \xi \\
u|_{t = 0} = u_0, 
\end{cases}
\qquad ( t, x) \in \R_+ \times \R^d, 
\end{equation}

\noi
where $\xi(t, x)$ denotes  a (Gaussian) space-time white noise on $\R_+ \times \R^d$
and $\phi$ is a bounded operator on $L^2(\R^d)$.
We say that $u$ is a solution to \eqref{SNLS1} if it satisfies 
the following mild formulation (= Duhamel formulation):
\begin{align}
u(t) = S(t) u_0+ i \int_0^t S(t - t') |u|^{p-1} u (t') dt' - i \int_0^t S(t - t') \phi \xi (dt'),
\label{SNLS2}
\end{align}

\noi
where $S(t) = e^{-it \Dl} $ denotes the linear Schr\"odinger propagator.
The last term on the right-hand side represents the effect of the stochastic forcing 
and is called the stochastic convolution, which we denote by $\Psi$:
\begin{align}
\Psi(t) :=  - i \int_0^t S(t - t') \phi \xi (dt').
\label{SNLS3}
\end{align}

\noi
In the following, we assume that $\phi \in \HS(L^2; H^s)$
for appropriate values of $s\geq 0$, 
namely, it is a Hilbert-Schmidt
operator from $L^2(\R^d)$ to $H^s(\R^d)$, guaranteeing that 
$\Psi \in C(\R_+; H^s(\R^d))$ almost surely \cite{DZ}.
See Section \ref{SEC:stoconv} for a further discussion on the stochastic convolution $\Psi$.
Previously, de Bouard-Debussche~\cite{DD} studied
\eqref{SNLS1} in the energy-subcritical setting\footnote{Namely, $ 1< p < 1 + \frac{4}{d-2}$ when $d \geq 3$
and $1 < p < \infty$ when $d = 1, 2$.
This guarantees that the scaling-critical regularity $s_\text{crit}$ defined in~\eqref{scaling2}
satisfies $s_\text{crit} < 1$.}
and proved its well-posedness in $H^1(\R^d)$,
assuming that $\phi \in \HS(L^2; H^1)$.
Our main goal in this paper is to present a simple construction
of  a unique local-in-time solution to  \eqref{SNLS1}
with a much rougher noise (and hence a rougher stochastic convolution)
than those considered in \cite{DD}.

Before discussing the well-posedness issue for SNLS \eqref{SNLS1}, 
let us first go over the local well-posedness theory for the following deterministic
nonlinear Schr\"odinger equations (NLS):
\begin{equation}
\begin{cases}\label{NLS1}
i \partial_t u =  \Delta u - |u|^{p-1} u  \\
u|_{t = 0} = u_0, 
\end{cases}
\qquad ( t, x) \in \R \times \R^d.
\end{equation}

\noi
The equation \eqref{NLS1} enjoys  the following dilation symmetry:
\begin{align}
 u(t, x) \longmapsto u^\ld(t, x) = \ld^{-\frac{2}{p-1}} u (\ld^{-2}t, \ld^{-1}x)
 \label{scaling1}
\end{align}

\noi
for  $\ld >0$.
Namely, if $u$ is a solution to \eqref{NLS1}, then the scaled function $u^\ld$ is also a solution to~\eqref{NLS1}
 with the rescaled initial data.
This dilation symmetry induces the following  scaling-critical Sobolev regularity:
 \begin{align}
 s_\text{crit} = \frac d2 - \frac{ 2}{p-1}
\label{scaling2}
 \end{align}
 
 \noi
such that the homogeneous $\dot{H}^{s_\text{crit}}(\R^d)$-norm is invariant
under the dilation symmetry.
This critical regularity $s_\text{crit}$ provides
a threshold regularity for well-posedness and ill-posedness
of \eqref{NLS1}.
Indeed, 
when $s \geq \max(s_\text{crit}, 0)$, 
the Cauchy problem \eqref{NLS1} is known to be locally well-posed in $H^s(\R^d)$
\cite{Tsu, CW}.\footnote{When $p$ is not an odd integer, 
we need to impose an extra assumption such as $p \geq [s] + 1$ due to the non-smoothness of the nonlinearity.
See also Remark \ref{REM:smooth}.
Note that this condition can be relaxed or eliminated in some situations.
See, for example, \cite{Kato}.
}
On the other hand,
it is known that NLS \eqref{NLS1} is ill-posed
in the scaling supercritical regime:  $s < s_\text{crit}$.
See \cite{CCT, Kishimoto,  O17}.

Let us now introduce two important critical regularities. 
When $s_\text{crit} = 0$, we 
say that the Cauchy problem \eqref{NLS1}
is {\it mass-critical}.  This corresponds to the case $p = 1 + \frac 4d$.
When $s_\text{crit} < 0$, i.e.~$p < 1 + \frac 4d$
 (and $s_\text{crit} > 0$, i.e.~$p > 1 + \frac 4d$, respectively),
we say that \eqref{NLS1} is mass-subcritical (and mass-supercritical, respectively).
When $s_\text{crit} = 1$, we 
say that the Cauchy problem \eqref{NLS1}
is {\it energy-critical}.  This corresponds to the case $p = 1 + \frac 4{d-2}$.
When $s_\text{crit} < 1$, i.e.~$p < 1 + \frac 4{d-2}$
 (and $s_\text{crit} > 1$, i.e.~$p > 1 + \frac 4{d-2}$, respectively),
we say that \eqref{NLS1} is energy-subcritical (and energy-supercritical, respectively).
In the following, we use the same terminology for SNLS \eqref{SNLS1}.

One of the main ingredients 
in establishing  local well-posedness of \eqref{NLS1}
is the following Strichartz estimates
\cite{Strichartz, Yajima, GV, KeelTao}:
\begin{equation}
\| S(t) u_0 \|_{L^q_t L^r_x (\R\times \R^d)} \leq C_{d, q, r} \|u_0\|_{L^2_x(\R^d)},
\label{Str1}
\end{equation}

\noi
which holds true for any Schr\"odinger admissible pair
$(q, r)$, satisfying 
\begin{equation}
\frac{2}{q} + \frac{d}{r} = \frac{d}{2}
\label{Str2}
\end{equation}

\noi
 with $2\leq q, r \leq \infty$
and $(q, r, d) \ne (2, \infty, 2)$.
In \cite{DD}, 
de Bouard-Debussche 
used the Strichartz estimates
to show  that the stochastic convolution $\Psi$
almost surely belongs to a right Strichartz space.
As a result, under the assumption that $\phi \in \HS(L^2; H^1)$, 
they proved that SNLS \eqref{SNLS1} is locally well-posed in $H^1(\R^d)$ in the energy-subcritical case:
$1 < p < 1 + \frac{4}{d-2}$ when $d \geq 3$ and $1 < p < \infty$ when $d = 1, 2$.\footnote{In \cite{DD}, 
they also proved global well-posedness of SNLS \eqref{SNLS1}.
The well-posedness issue for  SNLS with multiplicative noise was also considered in the same paper.
See also Cheung-Mosincat \cite{CM} for analogous well-posedness results 
of SNLS with additive and multiplicative noises
in the periodic setting.}
Now, let $s \geq \max (s_\text{crit}, 0)$.
Then,  by slightly modifying the argument in \cite{DD}
with Lemma \ref{LEM:stoconv} below, 
 it is easy to see  that 
SNLS \eqref{SNLS1} is locally well-posed\footnote{Once again, an extra assumption
such as $p \geq [s] + 1$ is needed when $p$ is not an odd integer.} in $H^s(\R^d)$, 
provided that $\phi\in \HS(L^2; H^s)$.
In particular, \eqref{SNLS1} is locally well-posed
in $L^2(\R^d)$
in the mass-(sub)critical case, 
provided that $\phi\in \HS(L^2; L^2)$.
Therefore, we focus our attention on the mass-supercritical case in the following.

We point out that so far we assumed that 
the noise had the same  spatial regularity 
as that of initial data.
On the one hand, the aforementioned ill-posedness results tell us
that we can not take (deterministic) initial data below the scaling-critical regularity $s_\text{crit}$.
On the other hand, we are allowed to take different regularities 
for initial data and the noise.
Indeed, in the following, we  treat rough stochastic noises that 
have regularities below the scaling critical regularity $s_\text{crit}$, 
while keeping (deterministic) initial data above
the scaling critical regularity.

\subsection{Main results}

In the following, we use $s_0$ 
and $s$ to denote the regularities of initial data $u_0$
and the noise (i.e.~$\phi \in \HS(L^2; H^s)$), respectively.
Our main goal is to lower the value of $s$, while keeping $s_0 \geq s_\text{crit}$.
In order to achieve this goal, we work within the $L_x^r$-framework, $r> 2$, 
by exploiting the following dispersive estimate:
\begin{equation}
\| S(t) u_0\|_{L^r_x (\R^d)} \le \frac {C_r}{|t|^{\frac{d}2 - \frac{d}r}}\|u_0\|_{L^{r'}_x(\R^d)}
\label{disp2}
\end{equation}

\noi
for any $2 \leq r \leq \infty$
with $\frac 1r + \frac 1{r'} = 1$.
Another key ingredient is the space-time integrability of the stochastic convolution.
By a small modification of the argument in \cite{DD}, we
show that,  the stochastic convolution $\Psi$ almost surely belongs
to 
\begin{align}
L^q([0, T]; W^{s, r}(\R^d))
\label{gain1}
\end{align}

\noi
for any $1\leq q < \infty$ 
and finite $r \geq 2$ such that $ r \leq \frac{2d}{d-2}$ when $d \ge 3$,
provided that $\phi \in \HS(L^2; H^s)$.
See Lemma~\ref{LEM:stoconv} below.
Note that the pair $(q, r)$ is no longer restricted to 
be Schr\"odinger admissible.
In particular, while keeping $r = \frac{2d}{d-2}$ and sending $q$ to $\infty$, 
we basically gain almost one spatial derivative.\footnote{Recall that $H^{s+1}(\R^d) \subset W^{s, \frac{2d}{d-2}}(\R^d)$.}
This allows us to prove the following improved local well-posedness result.\footnote{Here, 
we use the term ``well-posedness" in a loose sense as in \cite{BO96}.
See also Remark \ref{REM:WP}\,(ii).}

\begin{theorem}\label{THM:1}
\textup{(i) Energy-subcritical case:}\footnote{As we mentioned before, 
we assume that $(d, p)$ satisfies the mass-supercritical condition.} Let $d \geq 1$ and $1 + \frac 4d< p < \infty$.
When $d \geq 3$,  assume that $ p < 1 + \frac{4}{d-2}$ in addition.  

\smallskip

\noi
\textup{(i.a)}
Let  $s_0 \geq  \frac d2 - \frac{d}{p+1}$.
Then, 
given $u_0 \in H^{s_0}(\R^d) $, 
there exists a unique local-in-time solution $u$ to SNLS \eqref{SNLS1}, 
provided that $\phi \in \HS(L^2; L^2)$.
Moreover, the solution $u$ lies in the class:
\begin{align*}
 \Psi & + C([0, T]; L^{p+1}(\R^d))\cap C([0, T]; L^{2}(\R^d))\notag\\
&  \subset C([0, T]; L^{2}(\R^d)), 
\end{align*}

\noi
where $T = T_\o $ is almost surely positive.

\smallskip

\noi
\textup{(i.b)}
Let  $s_0 > s_\textup{crit}$.
Then, 
given $u_0 \in H^{s_0}(\R^d) $, 
there exists a unique local-in-time solution $u$ to SNLS \eqref{SNLS1}, 
provided that $\phi \in \HS(L^2; L^2)$.
Moreover, the solution $u$ lies in the class:
\begin{align*}
 \Psi & + L^q([0, T]; L^{p+1}(\R^d))\cap C([0, T]; L^{2}(\R^d))\notag\\
&  \subset C([0, T]; L^{2}(\R^d)), 
\end{align*}

\noi
where $q = q(d, p) > 2$ is finite and $T = T_\o $ is almost surely positive.

\smallskip

\noi
\textup{(ii) Energy-(super)critical case:} Let $d \geq 3$ and  $p \geq  1 + \frac{4}{d-2}$
be an odd integer.  
Fix $s_0 > s_\textup{crit}$ and $s > s_\textup{crit} - 1$.
Then, 
given $u_0 \in H^{s_0}(\R^d) $, 
there exists a unique local-in-time solution $u$ to SNLS \eqref{SNLS1}, 
provided that $\phi \in \HS(L^2; H^s)$.
Moreover, the solution $u$ lies in the class:
\begin{align*}
\Psi & + C([0, T]; W^{s_1, \frac{2d}{d-2}-\dl}(\R^d))
\cap 
C([0, T]; H^{s_1}(\R^d))\notag \\
& 
\subset C([0, T]; H^{s_1}(\R^d)), 
\end{align*}

\noi
where $s_1=  \min(s_0 - 1, s)$, 
 $\dl = \dl(s_1)> 0$ is sufficiently small, and $T = T_\o $ is almost surely positive.

\end{theorem}

The structure of the mild formulation \eqref{SNLS2}
states that any solution $u$ can be written as\footnote{The decomposition~\eqref{decomp1} 
is often referred to as the Da Prato-Debussche trick \cite{DPD}
in the field of stochastic parabolic  PDEs.
Such an idea also appears in 
McKean \cite{McKean} and 
Bourgain \cite{BO96} in the context of (deterministic) dispersive PDEs with random initial data, preceding \cite{DPD}. See also de Bouard-Debussche~\cite{DD} and Burq-Tzvetkov~\cite{BT1}.
 }
\begin{align}
u =  v + \Psi.
\label{decomp1}
\end{align}

\noi 
We then study the following fixed point problem for the residual term
$v := u - \Psi$:
\begin{align}
v(t) = S(t) u_0+  \int_0^t S(t - t') \N(v+\Psi) (t') dt' ,
\label{SNLS4}
\end{align}

\noi
where $\N(u) = i |u|^{p-1} u$.
In \cite{DD},  de Bouard-Debussche studied the fixed point problem~\eqref{SNLS4} for $v$
in terms of the standard $L_x^2$-based theory for 
NLS \eqref{NLS1}.
In particular, the solution $v$ to \eqref{SNLS4} was constructed 
in $C([0, T]; H^1(\R^d))$ intersected with an appropriate Strichartz space.
In the following, we instead work in the $L^r_x$-framework with $r > 2$
and directly solve the fixed point problem in $C([0, T]; W^{s_1, r}(\R^d))$\footnote{For Theorem \ref{THM:1} (i.b), 
we need to work in $L^q([0, T]; L^{p+1}(\R^d))$.}
by applying the dispersive estimate~\eqref{disp2}.

On the one hand, the spatial regularity $s_1$
of $v$  in Theorem \ref{THM:1}, i.e.~$s_1 = 0$ in~(i) and $s_1 = s_\textup{crit} - 1 + \eps$ in~(ii)
is below the scaling critical regularity $s_\text{crit}$ defined in \eqref{scaling2}
(when $\eps < 1$).
On the other hand, 
given any $ 1\leq r \leq \infty$, we can also consider  
the scaling-critical Sobolev regularity
adapted to the $L^r$-based Sobolev spaces:
\begin{align*}
 s_\text{crit}(r)  = \frac{d}{r} - \frac{2}{p-1}
\end{align*}

\noi
 such that the homogeneous $\dot{W}^{s_\textup{crit}(r), r}$-norm is invariant
under the dilation symmetry \eqref{scaling1}.
Note that we have 
$ s_\text{crit}(r) < s_\text{crit} =  s_\text{crit}(2)$
for $r > 2$.
For example, 
in the energy-(super)critical case, 
the gain of spatial integrability \eqref{gain1} of the stochastic convolution~$\Psi$
allows us to work in the $L^r_x$-based Sobolev space\footnote{For a technical reason, 
we need to take $r = \frac{2d}{d-2}-\dl$ for some small $\dl > 0$ in the proof of Theorem \ref{THM:1}.}
 with $r = \frac{2d}{d-2}$,
thus lowering the critical regularity from $s_\text{crit} = s_\text{crit}(2)$
to $ s_\text{crit}(r) = s_\text{crit} - 1$ with $r = \frac{2d}{d-2}$.
This heuristically explains 
 how the regularity threshold
$s_\textup{crit} - 1 $ appears in Theorem \ref{THM:1}\,(ii).
Moreover, note that, by working only within the $L^r$-based Sobolev space
with $s_1 > s_\text{crit}(r)$, 
we have made the problem
 {\it subcritical}.
Indeed, all the spatial function spaces such as $L^{p+1}(\R^d)$
appearing in Theorem \ref{THM:1} are subcritical in the sense described above.

\begin{remark} \label{REM:WP}\rm
(i) 
Our argument for proving Theorem \ref{THM:1} is of subcritical nature in the sense that the local existence
time $T$ depends on the $H^{s_0}$-norm of initial data
(and a space-time norm of the stochastic convolution).
It is possible to improve 
Theorem \ref{THM:1} (i.b) so that it also holds when $s_0 = s_\text{crit}$
by relying on the critical local well-posedness theory (in terms of initial data).
See Remark~\ref{REM:endpoint}.

\smallskip

\noi
(ii) Theorem \ref{THM:1} establishes existence of unique solutions
to \eqref{SNLS1}.
Note that the (spatial) regularity of the noise is rougher than that of the initial data in Theorem
\ref{THM:1} (i.a) and~(i.b).
As such, the solution inherits the rougher regularity of the noise 
and it only lies in $C([0, T]; L^2(\R^d))$.
The situation is slightly more subtle in 
Theorem \ref{THM:1} (ii).
Also, note that, in view of the aforementioned ill-posedness results, 
the map:  $(u_0, \phi \xi)\mapsto u$
is not continuous,  
when the noise has spatial regularity $s < s_\text{crit}$.
By the use of the Da Prato-Debussche trick, 
 however, 
 the  map sending an enhanced data set $(u_0, \Psi)$
 to a solution $u$ is continuous, 
where the stochastic convolution $\Psi$ is measured in an appropriate space-time function norm.
\end{remark}

\medskip

Our next goal is to  study the Cauchy problem~\eqref{SNLS1}
with random initial data and prove almost sure local well-posedness
for (random) initial data of  lower regularities.
More precisely,  given a function $u_0$ on $\R^d$, we consider 
a randomization of $u_0$ adapted to 
 the so-called  Wiener decomposition~\cite{W}
of the frequency space: $\R^d = \bigcup_{n \in \Z^d} Q_n$, 
where $Q_n$ is the unit cube centered at $n \in \Z^d$.

Let $\psi \in \mathcal{S}(\R^d)$ such that
\begin{equation*}
\supp \psi \subset [-1,1]^d
\qquad \text{and} \qquad\sum_{n \in \Z^d} \psi(\xi -n) \equiv 1
\ \text{ for any }\xi \in \R^d.
\end{equation*}

\noi
Then, given a function $u_0$ on $\R^d$, 
we have
\begin{align*}
u_0 = \sum_{n \in \Z^d} \psi(D-n) u_0,
\end{align*}

\noi
where $\psi(D-n)$ is defined by 
 $\psi(D-n)u_0(x)=\int_{\R^d} \psi (\xi-n)\ft u_0 (\xi)e^{ 2\pi ix\cdot \xi} d\xi$, 
 namely, 
 the Fourier multiplier operator with symbol $\ind_{Q_n}$ conveniently smoothed.
This decomposition leads to the following randomization of $u_0$
adapted to the Wiener decomposition.
Let $\{g_n\}_{n \in \Z^d}$ be a sequence of independent mean-zero complex-valued random variables
(with independent real and imaginary parts),
satisfying the following exponential moment bound:
\begin{align}
\E\big[ e^{\g_1 \Re g_n + \g_2 \Im g_n } \big] \leq e^{c( \g_1^2 + \g_2^2)}
\label{exp1}
\end{align}

\noi
for all $\g_1, \g_2 \in \R$ and  $n \in \Z^d$.
Note that \eqref{exp1} is satisfied by
standard complex-valued Gaussian random variables,
Bernoulli random variables,
and any random variables with compactly supported distributions.
We then  define the  Wiener randomization\footnote{It is also called 
the unit-scale randomization in \cite{DLM}.} of $u_0$ by
\begin{equation}
u_0^\omega  := \sum_{n \in \Z^d} g_n (\omega) \psi(D-n) u_0.
\label{gauss4}
\end{equation}

\noi
Given $u_0 \in H^{s}(\R^d)$, it is easy to see that its Wiener randomization $u_0^\o$
in \eqref{gauss4} lies in $ H^s(\R^d)$
almost surely.
One can also show that, under some non-degeneracy condition,  there is no smoothing upon randomization
in terms of differentiability;
see, for example, Lemma B.1 in \cite{BT1}.
The main point of the randomization \eqref{gauss4} is
its improved integrability.
For example, under
the assumption \eqref{exp1}, 
$u_0^\o$ almost surely belongs to $W^{s, r}(\R^d)$ for any finite $r \geq 2$.
Moreover, by restricting our attention to local-in-time intervals, 
the random linear solution $S(t) u_0^\o$
satisfies the Strichartz estimate \eqref{Str1}
for {\it any} finite $q, r \geq2$ almost surely.
See Lemma \ref{LEM:PStr1} below.
This gain of space-time integrability allows
us to take random initial data at the same low regularity as the stochastic forcing.

\begin{theorem}\label{THM:2}
\textup{(i) Energy-subcritical case:} 
Let $d$ and $p$ be as in Theorem \ref{THM:1} (i).
Then, 
given $u_0 \in L^2(\R^d) $, 
SNLS \eqref{SNLS1} is almost surely locally well-posed
with respect to the Wiener randomization $u_0^\o$ defined in \eqref{gauss4}, 
provided that $\phi \in \HS(L^2; L^2)$.
More precisely, 
there exists a unique local-in-time solution $u = u^\o$ to SNLS \eqref{SNLS1}
with $u|_{t = 0} = u_0^\o$
in the class:
\begin{align*}
S(t) u_0^\o + \Psi 
& + C([0, T]; L^{p+1}(\R^d))
\cap C([0, T]; L^{2}(\R^d))\notag\\
  & \subset C([0, T]; L^{2}(\R^d)), 
\end{align*}

\noi
where $T = T_\o $ is almost surely positive.

\smallskip

\noi
\textup{(ii) Energy-(super)critical case:} 
Let $d$ and $p$ be as in Theorem \ref{THM:1} (ii).
Moreover, let $s = s_\textup{crit} - 1+\eps$
for some small $\eps > 0$.
Then, 
given $u_0 \in H^{s}(\R^d) $, 
SNLS \eqref{SNLS1} is almost surely locally well-posed
with respect to the Wiener randomization $u_0^\o$
defined in \eqref{gauss4},  
provided that $\phi \in \HS(L^2; H^s)$.
More precisely, 
there exists a unique local-in-time solution $u = u^\o$  to SNLS \eqref{SNLS1}
with $u|_{ t= 0} = u_0^\o$ 
in the class: 
\begin{align*}
S(t) u_0^\o  + 
\Psi & +  C([0, T]; W^{s, \frac{2d}{d-2}-\dl}(\R^d))
\cap 
C([0, T]; H^s(\R^d))\notag\\
 & \subset
 C([0, T]; H^s(\R^d)), 
\end{align*}

\noi
where $\dl = \dl(s)> 0$ is sufficiently small 
 and $T = T_\o $ is almost surely positive.

\end{theorem}

In view of the probabilistic Strichartz estimates (Lemma \ref{LEM:PStr1}), 
we see that $\wt \Psi : = S(t)u_0^\o + \Psi$ solving
\begin{equation}
\begin{cases}\label{SNLS5}
i \partial_t \wt \Psi =   \Delta  \wt \Psi  + \phi \xi \\
\wt \Psi|_{t = 0} = u_0^\o, 
\end{cases}
\end{equation}

\noi
satisfies the same regularity properties, both in terms of differentiability
and integrability, as the stochastic convolution $\Psi$
in \eqref{SNLS3}.
Then, by decomposing $u$ as
\begin{align*}
u =   v + \wt \Psi, 
\end{align*}

\noi
Theorem \ref{THM:2} follows
from repeating  the argument in the proof of Theorem \ref{THM:1}.

We conclude this introduction with several remarks.

\begin{remark}\label{REM:smooth}\rm 
In Theorems \ref{THM:1} and \ref{THM:2}, 
we assumed that $p$ is an odd integer in the energy-(super)critical case.
One may apply the fractional chain rule
\cite{ChristW} and remove this restriction in certain situations.
For conciseness of the presentation, however, 
we do not pursue this direction in this paper.

\end{remark}

\begin{remark}\rm
In recent years, 
the well-posedness issue
of  the deterministic NLS \eqref{NLS1}
 with respect to the random initial data $u_0^\o$ in \eqref{gauss4}
has been studied intensively \cite{BOP1, BOP2, Bre, OOP, BOP3, DLM2}.
The main idea is to study the fixed point problem 
for the residual term $v = u - S(t) u_0^\o$,
utilizing (a variant of) the Fourier restriction norm method \cite{BO1, HHK, HTT1}
and carrying out rather tedious case-by-case analysis.
In a recent paper \cite{PW}, 
the second and third authors studied
the deterministic NLS \eqref{NLS1} with the random initial data $u_0^\o$ in  \eqref{gauss4}
by exploiting the dispersive estimate \eqref{disp2}.
In particular, they proved Theorem \ref{THM:2} above 
when $\phi = 0$, i.e.~when there is no stochastic noise.
This argument with the dispersive estimate
bypasses case-by-case analysis, which is closer in sprit
to the almost sure local well-posedness argument
for the nonlinear wave equations \cite{BT1, LM, POC, OP}.
See a survey paper \cite{BOP4} for a further discussion on the subject.

\end{remark}

\begin{remark}\rm 

In \cite{CPoc}, the second author with Cheung
 studied SNLS \eqref{SNLS1} on $\R^d$, $d\ge3$,  with the cubic nonlinearity ($p = 3$).
By adapting the argument in \cite{BOP2} for the deterministic NLS with random initial data, 
they proved local well-posedness of \eqref{SNLS1} with stochastic forcing below
the scaling-critical regularity, i.e. $\phi \in \HS(L^2; H^s)$ with $s < s_\text{crit}$.
Moreover, their work 
shows that
the residual part $v = u - \Psi$ lies in  
$C([0, T]; H^{s_\text{crit}}(\R^d))$.

\end{remark}

\smallskip

\noi
{\bf Notations:}
Given $T > 0$, we set $L^q_T B_x = L^q([0, T]; B(\R^d))$
and $C_T B_x = C([0, T]; B(\R^d))$, where $B(\R^d)$ denotes a Banach space of functions on $\R^d$.

\section{On the stochastic convolution}\label{SEC:stoconv}

In this section, we study the regularity properties of the stochastic convolution $\Psi$ in~\eqref{SNLS3}.
Let us first recall the definition of 
a cylindrical Wiener process $W$ on $L^2(\R^d)$.
Fix an orthonormal basis $\{e_n \}_{n \in \NB}$ of $L^2(\R^d)$.
Then, 
a cylindrical Wiener process $W$ on $L^2(\R^d)$ is defined by the following random Fourier series:
\begin{align*}
 W(t) & =  \sum_{n\in \NB}
\be_n (t) e_n, 
\end{align*}

\noi
where $\{ \be_n\}_{n \in \NB}$ is a family of  mutually independent complex-valued Brownian
motions.
In terms of the cylindrical Wiener process $W$, 
we can express the  stochastic convolution $\Psi$ in~\eqref{SNLS3}
as
\begin{align}
\Psi(t) =  - i \int_0^t S(t - t') \phi dW (t')
= - i  \sum_{n \in \NB} 
\int_0^t S(t - t') \phi(e_n)  d\beta_n (t').
\label{stoconv1}
\end{align}


By slightly modifying the argument in \cite{DD}, we have the following lemma.

\begin{lemma}\label{LEM:stoconv}
Suppose that $\phi \in \HS(L^2; H^s)$ for some $s \in \R$.
Then, the following statements hold almost surely:

\smallskip

\begin{itemize}
\item[\textup{(i)}]
$\Psi \in C(\R_+; H^s(\R^d))$, 

\smallskip

\item[\textup{(ii)}]  
Given any  $1 \leq q < \infty$ and finite $r \geq 2$ such that $ r \le \frac{2d}{d-2}$ when $d \geq 3$, 
we have
$\Psi \in L^q([0, T];  W^{s, r}(\R^d))$
for any $T > 0$.

\end{itemize}

\end{lemma}

Compare Part (ii) of Lemma \ref{LEM:stoconv} with \cite{DD}, where $(q, r)$ was restricted to be Schr\"odinger admissible.
It is this gain of integrability in time which allows us to prove Theorems \ref{THM:1} and \ref{THM:2}.

\begin{proof}
For (i), see \cite{DZ}.
Set $\jb{\nb} = \sqrt{1-\Dl}$.
Given $\phi \in \HS(L^2; H^s)$, 
let $\{ \phi_k\}_{k \in \NB} \subset \HS(L^2; H^{s+\s})$, $\s > \frac d2$, 
such that $\phi_k$ converges to $\phi$ in $\HS(L^2; H^s)$.
Then, letting $\Psi_k$ denote the stochastic convolution
in \eqref{stoconv1}
with $\phi$ replaced by $\phi_k$, 
we see that 
$\Psi_k$ converges to $\Psi$ in $C(\R_+; H^s(\R^d))$
and that 
$\jb{\nb}^s \Psi_k \in C(\R_{+}; H^\s(\R^d)) \subset C(\R_+; C(\R^d))$, 
where the inclusion follows from Sobolev's embedding theorem.

Fix $(t, x) \in \R_+ \times \R^d$.
Then, from \eqref{stoconv1}, 
we see that, as  a linear combination of independent Wiener integrals, 
$\jb{\nb}^s\Psi_k(t, x)$ is a mean-zero complex-valued Gaussian random variable
with variance $\s_k(t, x) = \|\jb{\nb}^s \Psi_k(t, x)\|_{L^2(\O)}^2$.
Recall that, for  a mean-zero complex-valued Gaussian random variable $g$
with variance $\s$, we have
\begin{align}
 \E\big[|g|^{2j} \big] = j! \cdot \s^j. 
 \label{stoconv1a}
\end{align}

\noi
Given $\rho \geq 2$,
let $\wt \rho$  denote the smallest even integer such that  $\wt \rho \geq \rho$.
Then, by H\"older's inequality and \eqref{stoconv1a}, 
we have\footnote{In fact, the following estimate
holds true:
\[\|\jb{\nb}^s \Psi_k(t, x)\|_{L^\rho(\O)}
\leq  \|\jb{\nb}^s \Psi_k(t, x)\|_{L^2(\O)}\]

\noi
Namely, there is no constant depending on $\rho \geq 2$.
See 
\cite[Theorem~I.22]{Simon}.
For our purpose, however, the elementary argument in \eqref{stoconv2} suffices.
}
\begin{align}
\|\jb{\nb}^s \Psi_k(t, x)\|_{L^\rho(\O)}
& \leq  \|\jb{\nb}^s \Psi_k(t, x)\|_{L^{\wt \rho}(\O)}\notag\\
& =  C_\rho \|\jb{\nb}^s \Psi_k(t, x)\|_{L^2(\O)}\notag\\
&  \sim  \bigg\| \bigg(
\int_0^t |S(t - t')\jb{\nb}^s  \phi_k(e_n)(x)|^2 dt' \bigg)^\frac{1}{2}\bigg\|_{\l^2_n}\notag\\
& =  \| S(\tau)\jb{\nb}^s  \phi_k(e_n)(x)  \|_{\l^2_nL^2_\tau([0, t])}
\label{stoconv2}
\end{align}

\noi
for any $\rho \geq 2$ and $(t, x) \in \R_+ \times \R^d$.

Now,  fix $1 \leq q < \infty$ and finite $r \geq 2$ such that $ r \le \frac{2d}{d-2}$ when $d \geq 3$.
Let $\wt q = \wt q(d, r) \geq 2 $ be the unique index such that 
$(\wt q, r)$ is Schr\"odinger admissible, satisfying \eqref{Str2}.
Then, for $\rho \geq \max (q, r)$, 
it follows from 
 Minkowski's integral inequality and \eqref{stoconv2}
 that 
\begin{align*}
\big\| \| \Psi_k\|_{L^q_T W^{s, r}_x}\big\|_{L^\rho(\O)}
& \leq 
\big\| \| \jb{\nb}^s \Psi_k(t, x) \|_{L^\rho(\O)}\big\|_{L^q_T L^r_x}\notag\\
& \leq C_\rho 
\big\| 
 \|S(\tau)\jb{\nb}^s  \phi_k(e_n)(x)  \|_{\l^2_nL^2_\tau([0, t])}\big\|_{L^q_T L^r_x}\notag\\
\intertext{By Minkowski's integral inequality (with $r \geq2$), 
H\"older's inequality in time,
and then applying the Strichartz estimate \eqref{Str1}, }
& \leq T^\frac{1}{q}  
\big\| 
 \|S(\tau)\jb{\nb}^s  \phi_k(e_n)  \|_{L^2_\tau([0, T]; L^r_x) }\big\|_{ \l^2_n}\notag\\
& \leq T^{\theta}
\big\| 
 \|S(\tau)\jb{\nb}^s  \phi_k(e_n)  \|_{L^{\wt q}_\tau([0, T]; L^r_x) }\big\|_{ \l^2_n}\notag\\
& \leq T^{\theta}
\big\| 
 \| \phi_k(e_n) \|_{H^s_x }\big\|_{ \l^2_n}\notag\\
&  = T^{\theta}
\| \phi_k\|_{\HS(L^2; H^s)} < \infty
\end{align*}

\noi
for some $\theta = \theta (q, \wt q) > 0$.
Similarly, we have
\begin{align*}
\big\| \| \Psi_k - \Psi_j \|_{L^q_T W^{s, r}_x}\big\|_{L^\rho(\O)}
\leq C T^{\theta} \| \phi_k - \phi_j \|_{\HS(L^2; H^s)}
\too 0, 
\end{align*}

\noi
as $k, j \to \infty$.
Namely, $\{ \Psi_k \}_{k \in \NB}$ is a Cauchy sequence
in $L^\rho(\O; L^q([0, T]; W^{s, r}(\R^d)))$.
By the uniqueness of the limit, 
we conclude that 
 $\{ \Psi_k \}_{k \in \NB}$ converges to $\Psi$ 
in $L^\rho(\O; L^q([0, T]; W^{s, r}(\R^d)))$.
In particular, we have
\begin{align*}
\big\| \| \Psi \|_{L^q_T W^{s, r}_x}\big\|_{L^\rho(\O)}
\leq C T^{\theta} \| \phi \|_{\HS(L^2; H^s)} < \infty.
\end{align*}

\noi
This proves (ii).
\end{proof}

\section{Proof of 
Theorems \ref{THM:1} and \ref{THM:2}}

In this section, we present the proofs of our main results (Theorems \ref{THM:1}
and \ref{THM:2}).
We first recall the following nonhomogeneous Strichartz estimate;
let $(q, r)$ and $(\wt q, \wt r)$ be Schr\"odinger admissible.
Then, we have
\begin{align}
\bigg\| \int_0^t S(t - t') F(t') dt' \bigg\|_{L^q_t L^r_x} \les \| F\|_{L^{\wt q'}_t L^{\wt r'}_x}, 
\label{Str3}
\end{align}

\noi
where ${\wt q}'$ and $\wt r'$ denote the H\"older conjugates of $\wt q$
and $\wt r$, respectively.

\subsection{Proof of Theorem \ref{THM:1}}

 Let $s_0, s \in \R$ to be specified later.
Given $u_0 \in H^{s_0}(\R^d)$ and $\phi \in \HS(L^2; H^s)$, we define $\G = \G_{u_0, \phi,\,  \xi}$ by 
\begin{equation*}
\G v(t): =S(t) u_0 +  \int_0^t S(t-t') \N (v+\Psi)(t') dt'.
\end{equation*}

\noi
Then,
we have  the following nonlinear estimates.

\begin{proposition}\label{PROP:NL1}
Let $d$ and $p$ be as in Theorem \ref{THM:1}.
We set
\begin{itemize}
\item[\textup{(i.a)}] $s_0\geq \frac{d}{2} - \frac{d}{p+1}$, $s_1 = 0$,  $r = p+1$ in the energy-subcritical case, 

\item[\textup{(ii)}] $s_0 -1 \geq s_1 >  s_\textup{crit}-1$ 
 and $r = \frac{2d}{d-2}-\dl$ 
for some small $\dl = \dl(s_1) > 0$ in the energy-(super)critical case.

\end{itemize}

\noi
Then, the following estimates hold for some $q\gg 1$:
\begin{align}
\| \G v \|_{C_T W^{s_1, r}_x}
&\les \|u_0\|_{H^{s_0}} 
 + T^\theta\Big(
\| v\|_{C_T W_x^{s_1, r}}^p
+ \| \Psi\|_{L^q_T W_x^{s_1, r}}^p\Big), \notag
\\
\| \G v_1 - \G v_2  \|_{C_T W^{s_1, r}_x}
&\les 
  T^\theta\Big(
\| v_1\|_{C_T W_x^{s_1, r}}^{p-1}
+ \| v_2\|_{C_T W_x^{s_1, r}}^{p-1} \notag\\
& \hphantom{XXXXXXXXX}
+ \| \Psi\|_{L^q_T W_x^{s_1, r}}^{p-1}\Big)
\| v_1 - v_2\|_{C_T W_x^{s_1, r}}
\label{nonlin0b}
\end{align}

\noi
for all $v, v_1, v_2 \in C([0,T]; W^{s_1,r}(\R^d))$ and $T>0$.
Moreover, we have $\G v \in C([0,T]; H^{s_1}(\R^d))$ 
for all $v \in C([0,T]; W^{s_1,r}(\R^d))$,
 $\Psi \in L^q([0,T]; W^{s_1,r}(\R^d))$, and $T>0$.

\end{proposition}

Once we prove Proposition \ref{PROP:NL1}, 
Theorem \ref{THM:1} (i.a) and (ii) follow
from a standard contraction argument with Lemma \ref{LEM:stoconv}
(with  $s = s_\textup{crit} - 1+\eps$ for Theorem \ref{THM:1}\,(ii)).

\begin{proof}
We first consider the energy-subcritical case (i.a).
By Sobolev's inequality $H^{s_0}(\R^d) \subset L^{p+1}(\R^d)$
and the dispersive estimate \eqref{disp2}, we have 
\begin{align}
\| \G v \|_{C_TL_x^{p+1}}
 &\les \|S(t) u_0\|_{C_T H^{s_0}}  + \sup_{t\in [0,T]} \int_0^t \frac1{|t-t'|^{\frac{d}{2} - \frac{d}{p+1}}} 
\|   \N(v+\Psi) (t')\|_{L_x^{\frac{p+1}{p}}} dt'\notag\\
&\les \|u_0\|_{H^{s_0}} + T^\theta \|   v+\Psi \|_{L^q_T L_x^{p+1}}^p \notag\\
&\les \|u_0\|_{H^{s_0}} 
+ T^\theta \Big(
\|  v \|_{C_T L_x^{p+1}}^p + \|  \Psi \|_{L^q_T L_x^{p+1}}^p \Big)
\label{nonlin1}
\end{align}

\noi
for some $q\gg 1$ and small $\theta > 0$, 
provided that ${s_0} \geq \frac d2 - \frac d {p+1}$
and $\frac{d}{2} - \frac{d}{p+1} <1$,
namely $p < 1 + \frac{4}{d-2}$.
When $p$ is an odd integer, a similar computation yields the following difference estimate:
\begin{align}
\| \G v_1 - \G v_2 \|_{C_TL_x^{p+1}}
 &\les T^\theta \|   \N(v_1+\Psi) -  \N(v_2+\Psi) \|_{L^q_T L_x^{\frac{p+1}{p}}}\notag\\
&\les   T^\theta \Big(
\|  v_1 \|_{C_T L_x^{p+1}}^{p-1}
+ \|  v_2 \|_{C_T L_x^{p+1}}^{p-1}
 + \|  \Psi \|_{L^q_T L_x^{p+1}}^{p-1} \Big)
\|  v_1 - v_2 \|_{C_T L_x^{p+1}}.
\label{nonlin2}
\end{align}

\noi
Next, we consider the case when   $p >1$ is not an odd integer.
By the mean value theorem, we have
\begin{align}
\N(v_1 + \Psi) - \N(v_2 + \Psi)  
= \int_0^1  \Big\{ & \dd_z \N(v_2 + \Psi + \theta(v_1 - v_2)) (v_1 - v_2) \notag\\
& + \dd_{\cj{z}} \N(v_2 + \Psi + \theta(v_1 - v_2)) (\cj{v_1 - v_2})\Big\} d\theta.
\label{nonlin3}
\end{align}

\noi
With $\N(z) = i|z|^{p-1} z$, we have
\begin{align}
\textstyle \dd_z \N(z) = i \frac{p+1}{2} |z|^{p-1}\qquad \text{and}
\qquad  \dd_{\cj z} \N(z) = i \frac{p-1}{2} |z|^{p-1}\frac{z^2}{|z|^2}.
\label{nonlin4}
\end{align}

\noi
Then, by repeating the computation above with \eqref{nonlin3} and \eqref{nonlin4}, 
we obtain \eqref{nonlin2}.
Given $p + 1 < \frac {2d}{d-2}$, let $(\wt q, p+1)$ be Schr\"odinger admissible.
Then, it follows from \eqref{Str3} that 
\begin{align*}
\| \G v - S(t) u_0 \|_{C_TL_x^2}
 &\les
   \|   \N(v+\Psi) \|_{L^{\wt q'}_TL_x^{\frac{p+1}{p}}} \notag\\
&\les 
 T^\theta \Big(
\|  v \|_{C_T L_x^{p+1}}^p + \|  \Psi \|_{L^q_T L_x^{p+1}}^p \Big)
\end{align*}

\noi
for some $q\gg 1$ and small $\theta > 0$.
Hence, we conclude that 
 $\G v \in C([0,T]; L^2(\R^d))$ 
for all $v \in C([0,T]; L^{p+1}(\R^d))$ and $T>0$.

Next, we consider the energy-(super)critical case (ii): $p \geq 1 + \frac{4}{d-2}$ when $d \geq 3$.
Since   $s_0 -  1 \geq  s_1> s_\text{crit} - 1 = \frac{d-2}{2} - \frac{2}{p-1}$, 
we can choose $\dl > 0$ sufficiently small such that 
\begin{align}
H^{s_0}(\R^d) \subset W^{s_1,r}(\R^d)\subset L^{\frac{(p-1)r}{r-2}}(\R^d),
\label{Sob1}
\end{align}

\noi
where $ r = \frac{2d}{d-2} - \dl$.
Since $p$ is an odd integer, the nonlinearity $\N(u)$ is algebraic 
and hence we can apply
the fractional Leibniz rule.
Then, proceeding with 
the dispersive estimate~\eqref{disp2}, the fractional Leibniz rule, 
and \eqref{Sob1}, we have
\begin{align}
\| \G v \|_{C_T W^{s_1, r}_x}
 &\les \|S(t) u_0\|_{C_T H^{s_0}}  
 + T^\theta
\|  \jb{\nb}^{s_1} \N(v+\Psi)\|_{L^\frac{q}{p}_T L_x^{r'}}\notag\\
&\les \|u_0\|_{H^{s_0}} 
 + T^\theta
\|  \jb{\nb}^{s_1} (v+\Psi)\|_{L^q_T L_x^{r}}
\|  v+\Psi\|_{L^q_T L_x^{\frac{(p-1)r}{r-2}}}^{p-1}\notag\\
&\les \|u_0\|_{H^{s_0}} 
 + T^\theta
\| v+\Psi\|_{L^q_T W_x^{s_1, r}}^p\notag\\
&\les \|u_0\|_{H^{s_0}} 
 + T^\theta\Big(
\| v\|_{C_T W_x^{s_1, r}}^p
+ \| \Psi\|_{L^q_T W_x^{s_1, r}}^p\Big)
\label{nonlin5}
\end{align}

\noi
for some $q \gg1 $ and small $\theta > 0$.
The difference estimate \eqref{nonlin0b} follows in a similar manner.
Given $ r = \frac{2d}{d-2} - \dl$, let $(\wt q, r)$ be Schr\"odinger admissible.
Then, proceeding as in \eqref{nonlin5} with~\eqref{Str3}, we have
\begin{align*}
\| \G v - S(t) u_0 \|_{C_T H^{s_1}_x}
 &\les 
\|  \jb{\nb}^{s_1} \N(v+\Psi)\|_{L^{\wt q'}_T L_x^{r'}}\notag\\
&\les 
  T^\theta\Big(
\| v\|_{C_T W_x^{s_1, r}}^p
+ \| \Psi\|_{L^q_T W_x^{s_1, r}}^p\Big)
\end{align*}

\noi
for some $q \gg1 $ and small $\theta > 0$.
This shows 
 $\G v \in C([0,T]; H^{s_1}(\R^d))$ 
for all $v \in C([0,T]; W^{s_1,  r}(\R^d))$ and $T>0$.
\end{proof}

Similarly, 
Theorem \ref{THM:1} (i.b)
follows from the following proposition.
In order to control the linear solution at a lower regularity, 
we apply the Strichartz estimate \eqref{Str1}.

\begin{proposition}\label{PROP:NL2}
Let $d$ and $p$ be as in Theorem \ref{THM:1} (i.b)
and $s_0> s_\textup{crit}$.
Then, the following estimates hold for some $q\gg 1$:
\begin{align}
\| \G v \|_{L^q_T L^{p+1}_x}
&\les \|u_0\|_{H^{s_0}} 
 + T^\theta\Big(
\| v\|_{L^q_T L^{p+1}_x}^p
+ \| \Psi\|_{L^q_T L^{p+1}_x}^p\Big), \notag
\\
\| \G v_1 - \G v_2  \|_{L^q_T L^{p+1}_x}
&\les 
  T^\theta\Big(
\| v_1\|_{L^q_T L^{p+1}_x}^{p-1}
+ \| v_2\|_{L^q_T L^{p+1}_x}^{p-1}
+ \| \Psi\|_{L^q_T L^{p+1}_x}^{p-1}\Big)
\| v_1 - v_2\|_{L^q_T L^{p+1}_x}
\label{nonlin6}
\end{align}

\noi
for all $v, v_1, v_2 \in L^q([0,T]; L^{p+1}(\R^d))$ and $T>0$.
Moreover, we have $\G v \in C([0,T]; L^2(\R^d))$ 
for all $v \in L^q([0,T]; L^{p+1}(\R^d))$ and $T>0$.

\end{proposition}

\begin{proof}
Given $(d, p)$, fix $q \geq 2$ such that 
$\frac{1}{q} + 1 = \Big( \frac{d}{2} - \frac{d}{p+1}+\eps \Big) + \frac{p}{q}$
for some small $\eps > 0$.
Furthermore, let $r \geq 2$ such that $(q, r)$ is Schr\"odinger admissible.
Then, by Sobolev's inequality, we have 
\begin{align}
W^{s_0, r}(\R^d) \subset L^{p+1}(\R^d).
\label{Sob2}
\end{align}

\noi
By proceeding as in \eqref{nonlin1} with 
\eqref{Sob2}, the dispersive estimate \eqref{disp2}, 
the Strichartz estimate~\eqref{Str1},  and Young's inequality, we have
\begin{align}
\| \G v \|_{L^q_TL_x^{p+1}}
 &\les \|S(t) u_0\|_{L^q_T W^{s_0, r}}  + \bigg\| \int_0^t \frac1{|t-t'|^{\frac{d}{2} - \frac{d}{p+1}}} 
\|   \N(v+\Psi) (t')\|_{L_x^{\frac{p+1}{p}}} dt' \bigg\|_{L^q_T}\notag\\
&\les \|u_0\|_{H^{s_0}} 
+ T^\theta \Big(
\|  v \|_{L^q_T L_x^{p+1}}^p + \|  \Psi \|_{L^q_T L_x^{p+1}}^p \Big)
\label{ZZ1}
\end{align}

\noi
for some $\theta > 0$.
The difference estimate \eqref{nonlin6} follows in a similar manner.
Given $p + 1 < \frac {2d}{d-2}$, 
let $(\wt q, p+1)$ be Schr\"odinger admissible.
Then, from the mass-supercritical condition: $p > 1 + \frac 4d$, 
we see that $q > p\, \wt q'$.
Hence, 
it follows from \eqref{Str3} that 
\begin{align*}
\| \G v - S(t) u_0 \|_{C_TL_x^2}
 &\les   \|   \N(v+\Psi) \|_{L^{\wt q'}_TL_x^{\frac{p+1}{p}}} \notag\\
&\les  T^\theta \Big(
\|  v \|_{L^q_T L_x^{p+1}}^p + \|  \Psi \|_{L^q_T L_x^{p+1}}^p \Big)
\end{align*}

\noi
for some small $\theta > 0$.
\end{proof}

\begin{remark}\label{REM:endpoint}\rm
In \eqref{ZZ1}, it is possible to  apply Hardy-Littlewood-Sobolev's inequality
instead of Young's inequality since $q < \infty$.
Namely, proceeding with $\eps = 0$, 
Hardy-Littlewood-Sobolev's inequality gives
\begin{align*}
\| \G v \|_{L^q_TL_x^{p+1}}
 &\les \Big(\|S(t) u_0\|_{L^q_T W^{s_0, r}} 
 + \|  \Psi \|_{L^q_T L_x^{p+1}}^p \Big)
 + \|  v \|_{L^q_T L_x^{p+1}}^p. 
\end{align*}

\noi
The difference estimate \eqref{nonlin6} also holds
without the  $T^\theta$-factor.
Then, we can carry out a contraction argument by 
making $T  = T_\o> 0$ sufficiently small such that 
\[\|S(t) u_0\|_{L^q_T W^{s_0, r}} 
 + \|  \Psi \|_{L^q_T L_x^{p+1}}^p \ll 1,\]
 
 \noi
 as in the mass-critical local well-posedness theory for NLS \eqref{NLS1}, 
 and prove local well-posedness of \eqref{SNLS1} even when $s_0 = s_\text{crit}$.
As the argument is standard, we omit details.
Note that we  measure the stochastic convolution $\Psi$ only with the subcritical $L^{p+1}_x$-norm.

\end{remark}

\subsection{Proof of Theorem \ref{THM:2}}
Given $u_0$ on $\R^d$, let $u_0^\o$ be its Wiener randomization defined in~\eqref{gauss4}.
Then,  we define $\wt \G = \wt \G_{u_0^\o, \phi, \,\xi}$ by 
\begin{equation*}
\wt \G v(t): =  \int_0^t S(t-t') \N (v+\wt \Psi)(t') dt',
\end{equation*}

\noi
where $\wt \Psi$ is the stochastic convolution defined in \eqref{SNLS5}
such that $\wt \Psi|_{t = 0} = u_0^\o$.
Then, by proceeding as in 
the proof of  Proposition \ref{PROP:NL1}, 
we obtain the following nonlinear estimates.

\begin{proposition}\label{PROP:NL3}
Let $d$ and $p$ be as in Theorem \ref{THM:1}.
We set
\begin{itemize}
\item[\textup{(i)}] $s = 0$ and $r = p+1$ in the energy-subcritical case, 

\item[\textup{(ii)}] $s > s_\textup{crit} - 1$ and $r = \frac{2d}{d-2}-\dl$ 
for some small $\dl = \dl(s)> 0$ in the energy-(super)critical case.

\end{itemize}

\noi
Then, the following estimates hold for some $q\gg 1$:
\begin{align*}
\| \wt \G v \|_{C_T W^{s, r}_x}
&\les T^\theta\Big(
\| v\|_{C_T W_x^{s, r}}^p
+ \| \wt \Psi\|_{L^q_T W_x^{s, r}}^p\Big), 
\\
\| \wt \G v_1 - \wt \G v_2  \|_{C_T W^{s, r}_x}
&\les 
  T^\theta\Big(
\| v_1\|_{C_T W_x^{s, r}}^{p-1}
+ \| v_2\|_{C_T W_x^{s, r}}^{p-1}
+ \| \wt \Psi\|_{L^q_T W_x^{s, r}}^{p-1}\Big)
\| v_1 - v_2\|_{C_T W_x^{s, r}}
\end{align*}

\noi
for all $v, v_1, v_2 \in C([0,T]; W^{s,r}(\R^d))$.
\end{proposition}

Next, let us state the following probabilistic Strichartz estimates.
See \cite{BOP1} for the proof.

	\begin{lemma}\label{LEM:PStr1}
Let $s \in \R$. Given $u_0$ on $H^s(\R^d)$, 
let $u_0^\o$ be its randomization defined in \eqref{gauss4}, satisfying~\eqref{exp1}.
Then, the following statements hold almost surely:

\smallskip

\begin{itemize}
\item[\textup{(i)}]
$S(t)u_0^\o  \in C(\R; H^s(\R^d))$, 

\smallskip
\item[\textup{(ii)}]  
Given finite   $q, r \geq 2$, 
we have
$S(t) u_0^\o \in L^q([0, T];  W^{s, r}(\R^d))$
for any $T > 0$.

\end{itemize}

\end{lemma}

In particular, Lemmas \ref{LEM:stoconv} and \ref{LEM:PStr1} 
state that the new stochastic convolution $\wt \Psi = S(t) u_0^\o+ \Psi$
also satisfies the conclusion of Lemma \ref{LEM:stoconv}.
Therefore, together with this observation,
  Proposition~\ref{PROP:NL3} implies
 Theorem \ref{THM:2}.

\begin{acknowledgment}

\rm 
T.\,O.~was supported by the European Research Council (grant no.~637995 ``ProbDynDispEq'').

\end{acknowledgment}


\begin{thebibliography}{99}		



\bibitem{BOP1}
\'A.~B\'enyi, T.~Oh, O.~Pocovnicu,
{\it Wiener randomization on unbounded domains
and an application to almost sure well-posedness  of NLS},
Excursions in Harmonic Analysis, Volume 4, 3-25, Appl. Numer. Harmon. Anal., Birkh\"auser/Springer, New York, 2015. 

\bibitem{BOP2}
\'A.~B\'enyi, T.~Oh, O.~Pocovnicu,
{\it On the probabilistic Cauchy theory of the cubic nonlinear Schr\"odinger equation on $\R^d$, $d \ge 3$},
Trans. Amer. Math. Soc. Ser. B 2 (2015), 1--50. 

\bibitem{BOP3}
\'A.~B\'enyi, T.~Oh, O.~Pocovnicu,
{\it Higher order expansions for the probabilistic local Cauchy theory of the cubic nonlinear Schr\"odinger equation on $\R^3$}, 
to appear in Trans. Amer. Math. Soc.

%


\bibitem{BOP4}
\'A.~B\'enyi, T.~Oh, O.~Pocovnicu,
{\it On the probabilistic Cauchy theory for nonlinear dispersive PDEs}, 
 to appear in Landscapes of Time-Frequency Analysis, Appl. Numer. Harmon. Anal.


\bibitem{BO1} J.~Bourgain,
{\it Fourier transform restriction phenomena for certain lattice subsets and applications to nonlinear evolution equations. I. Schr\"{o}dinger equations}, Geom. Funct. Anal. 3 (1993), 107--156.




\bibitem{BO96} J. Bourgain, {\it Invariant measures for the $2D$-defocusing nonlinear Schr\"odinger equation,}
  Comm. Math. Phys.  176  (1996),  no. 2, 421--445.

\bibitem{Bre}
J.~Brereton, 
{\it Almost sure local well-posedness for the supercritical quintic NLS}, 
Tunisian J. Math.
  1  (2019), no. 3, 427--453.
%

\bibitem{BT1}
N.~Burq, N.~Tzvetkov,
{\it Random data Cauchy theory for supercritical wave equations. I. Local theory,}
Invent. Math. 173 (2008), no. 3, 449--475.


\bibitem{CW}
T.~Cazenave, F.~Weissler,
{\it The Cauchy problem for the critical nonlinear Schr\"odinger equation in $H^s$,}
 Nonlinear Anal. 14 (1990), no. 10, 807--836.


\bibitem{CM}
K.~Cheung, R.~Mosincat,
{\it Stochastic nonlinear Schr\"odinger equations on tori,}
to appear in Stoch. Partial Differ. Equ. Anal. Comput. 
%


\bibitem{CPoc}
K.~Cheung, O.~Pocovnicu, 
{\it  On the local well-posedness  of the 
stochastic cubic nonlinear Schr\"odinger equation on $\R^d$, $d\geq 3$,
with supercritical noise},
preprint.


\bibitem{CCT} M. Christ, J. Colliander, T. Tao, 
{\it Ill-posedness for nonlinear Schr\"odinger and wave equations}, 
arXiv:math/0311048 [math.AP].

\bibitem{ChristW}
M.~Christ, M.~Weinstein, 
{\it Dispersion of small amplitude solutions of the generalized
Korteweg-de Vries equation}, J. Funct. Anal. 100 (1991), 87--109.



\bibitem{DPD}
G.~Da Prato, A.~Debussche, 
{\it Two-dimensional Navier-Stokes equations driven by a space-time white noise}, 
J. Funct. Anal. 196 (2002), no. 1, 180--210.






\bibitem{DZ}
 G.~Da Prato, J.~Zabczyk, 
 {\it Stochastic equations in infinite dimensions,} Second edition. Encyclopedia of Mathematics and its Applications, 152. Cambridge University Press, Cambridge, 2014. xviii+493 pp.




\bibitem{DD}
A.~de Bouard, A.~Debussche, 
{\it The stochastic nonlinear Schr\"odinger equation in $H^1$,} Stochastic Anal. Appl. 21 (2003), no. 1, 97--126.



\bibitem{DLM}
B.~Dodson, J.~L\"uhrmann, D.~Mendelson, 
{\it Almost sure scattering for the 4D energy-critical defocusing nonlinear wave equation with radial data}, 
to appear in Amer. J. Math.



\bibitem{DLM2}
B.~Dodson, J.~L\"uhrmann, D.~Mendelson,
{\it Almost sure local well-posedness and scattering for the 4D cubic nonlinear Schr\"odinger equation},
arXiv:1802.03795 [math.AP].



\bibitem{GV}
J.~Ginibre, G.~Velo, {\it  Smoothing properties and retarded estimates for some dispersive evolution equations},
Comm. Math. Phys. 144 (1992), no. 1, 163--188.



\bibitem{HHK}
M.~Hadac, S.~Herr, H.~Koch,
{\it Well-posedness and scattering for the KP-II equation in a critical space,}
 Ann. Inst. H. Poincar\'e Anal. Non Lin\'eaire  26 (2009), no. 3, 917--941.
{\it  Erratum to ``Well-posedness and scattering for the KP-II equation in a critical space''},
Ann. Inst. H. Poincar\'e Anal. Non Lin\'eaire  27 (2010), no. 3, 971--972.

\bibitem{HTT1}S.~Herr, D.~Tataru, N.~Tzvetkov,
{\it Global well-posedness of the energy critical nonlinear Schr\"odinger equation with small initial data in $H^1(\T^3)$}, Duke Math. J.  159 (2011), 329--349.



\bibitem{Kato}
T.~Kato, 
{\it On nonlinear Schr\"odinger equations,} 
Ann. Inst. H. Poincar\'e Phys. Th\'eor. 46 (1987), no. 1, 113--129.


\bibitem{KeelTao}
M.~Keel, T.~Tao, {\it Endpoint Strichartz estimates},
Amer. J. Math. 120 (1998), no. 5, 955--980.


\bibitem{Kishimoto}
N.~Kishimoto, 
{\it A remark on norm inflation for nonlinear Schr\"odinger equations},
arXiv:1806.10066 [math.AP].

\bibitem{LM}
J.~L\"uhrmann, D.~Mendelson,
{\it Random data Cauchy theory for nonlinear wave equations of power-type on $\mathbb{R}^3$},
 Comm. Partial Differential Equations  39 (2014), no. 12, 2262--2283.




\bibitem{McKean}
H.P.~McKean, 
{\it Statistical mechanics of nonlinear wave equations. IV. Cubic Schr\"odinger,} 
 Comm. Math. Phys. 168 (1995), no. 3, 479--491. 
 {\it Erratum: Statistical mechanics of nonlinear wave equations. IV. Cubic Schr\"odinger}, Comm. Math. Phys. 173 (1995), no. 3, 675.



\bibitem{O17}
T.~Oh, 
{\it A remark on norm inflation with general initial data for the cubic nonlinear Schr\"odinger equations in negative Sobolev spaces}, 
Funkcial. Ekvac. 60 (2017), 259--277. 




\bibitem{OOP}
T.~Oh, M.~Okamoto, O.~Pocovnicu,
{\it On the probabilistic  well-posedness of the  nonlinear Schr\"{o}dinger equations
with non-algebraic nonlinearities}, 
 	arXiv:1708.01568 [math.AP].



\bibitem{OP}
T.~Oh, O.~Pocovnicu, 
{\it Probabilistic global well-posedness of the energy-critical defocusing quintic nonlinear wave equation on $\R^3$}, 
J. Math. Pures Appl.  105 (2016), 342--366. 



\bibitem{POC}
O.~Pocovnicu,
{\it Probabilistic  global well-posedness of the energy-critical defocusing  cubic nonlinear
wave equations on $\R^4$}, 
 J. Eur. Math. Soc. (JEMS)  19 (2017), 2321--2375. 



\bibitem{PW}
O.~Pocovnicu, Y.~Wang, 
{\it An $L^p$-theory for almost sure local well-posedness
of the nonlinear Schr\"odinger equations}, 
C. R. Math. Acad. Sci. Paris.
356 (2018), no. 6, 637--643. 



\bibitem{Simon}
B.~Simon, 
{\it  The $P(\varphi)_2$ Euclidean (quantum) field theory,} Princeton Series in Physics. Princeton University Press, Princeton, N.J., 1974. xx+392 pp.




\bibitem{Strichartz} 
R.~Strichartz,
{\it Restrictions of Fourier transforms to quadratic surfaces and decay of solutions of wave equations},
Duke Math. J. 44 (1977), no. 3, 705--714.



\bibitem{Tsu} Y.~Tsutsumi, 
{\it $L^2$-solutions for nonlinear Schr\"odinger equations and nonlinear groups}, 
Funkcial. Ekvac.
30 (1987), no. 1, 115--125.



\bibitem{W}
N.~Wiener, {\it Tauberian theorems}, Ann. of Math. 33 (1932), no. 1, 1--100.




\bibitem{Yajima}
K.~Yajima,  {\it Existence of solutions for Schr\"odinger evolution equations},
Comm. Math. Phys. 110 (1987), no. 3, 415--426.

























%
%
%








\end{thebibliography}
\end{document}